\documentclass[11pt]{article}
\usepackage{amsmath}
\usepackage{amsthm,thmtools}
\usepackage[english]{babel}
\usepackage[T1]{fontenc}
\usepackage[utf8]{inputenc}
\usepackage{graphicx}
\usepackage[margin=1.25in]{geometry}
\usepackage{amsfonts}
\usepackage{amssymb}
\usepackage{color}
\usepackage{esvect}
\usepackage{pgf,tikz}
\usepackage{tkz-euclide}
\usepackage[parfill]{parskip}
\usepackage{mathrsfs}
\usepackage[bookmarks=false,hidelinks]{hyperref}
\usetikzlibrary{arrows}
\def\ex{\mathrm{ex}}

\newtheorem{theorem}{Theorem}[section]
\newtheorem{lemma}[theorem]{Lemma}
\newtheorem{claim}[theorem]{Claim}
\newtheorem{proposition}[theorem]{Proposition}
\newtheorem{conjecture}[theorem]{Conjecture}

\theoremstyle{definition}

\newtheorem{definition}[theorem]{\bf Definition}

\title{On the generalized Tur\'an problem for odd cycles}

\author{Csongor Beke\thanks{Trinity College, University of Cambridge, United Kingdom. Research supported by the Trinity College Summer Studentship Scheme. Email: \textbf{cb2138@cam.ac.uk}.}
	\and
Oliver Janzer\thanks{Department of Pure Mathematics and Mathematical Statistics, University of Cambridge, United Kingdom. Research supported by a fellowship at Trinity College. Email: \textbf{oj224@cam.ac.uk}.}}

\date{}
 
\begin{document}

\maketitle

\begin{abstract}
    In 1984, Erd\H os conjectured that the number of pentagons in any triangle-free graph on $n$ vertices is at most $(n/5)^5$, which is sharp by the balanced blow-up of a pentagon. This was proved by Grzesik, and independently by Hatami, Hladk\'y, Kr\'al', Norine and Razborov. As an extension of this result for longer cycles, we prove that for each odd $k\geq 7$, the balanced blow-up of $C_k$ (uniquely) maximises the number of $k$-cycles among $C_{k-2}$-free graphs on $n$ vertices, as long as $n$ is sufficiently large. We also show that this is no longer true if $n$ is not assumed to be sufficiently large. Our result strengthens results of Grzesik and Kielak who proved that for each odd $k\geq 7$, the balanced blow-up of $C_k$ maximises the number of $k$-cycles among graphs with a given number of vertices and no odd cycles of length less than $k$.

    We further show that if $k$ and $\ell$ are odd and $k$ is sufficiently large compared to $\ell$, then the balanced blow-up of $C_{\ell+2}$ does not asymptotically maximise the number of $k$-cycles among $C_{\ell}$-free graphs on $n$ vertices. This disproves a conjecture of Grzesik and Kielak.
\end{abstract}

\section{Introduction}

The Tur\'an problem is one of the central topics in Graph Theory. It is concerned with estimating, for a graph $H$ and positive integer $n$, the maximum number of edges that an $n$-vertex, $H$-free graph can have. A generalization of this problem was introduced by Alon and Shikhelman \cite{AS16}. For graphs $T$ and $H$ and a positive integer $n$, they defined $\ex(n,T,H)$ to be the maximum possible number of copies of $T$ in an $H$-free graph on $n$ vertices. This function became known as the \emph{generalized Tur\'an number} of $T$ and $H$. The name reflects the fact that when $T=K_2$, then $\ex(n,T,H)$ is just the usual Tur\'an number $\ex(n,H)$. Although the systematic study of generalized Tur\'an numbers was only initiated recently, many special cases had already been studied for decades. The first such result (in which $T$ is different from $K_2$) was obtained in 1949 by Zykov \cite{Zyk49}, who determined $\ex(n,K_t,K_r)$ for all $t<r$. The same result was proven later but independently by Erd\H os \cite{Erd62} and by Roman \cite{Rom76}.

The other instance in which the problem has a long history is when $T$ and $H$ are odd cycles. In 1984, Erd\H os \cite{Erd84} conjectured that $\ex(n,C_5,C_3)\leq (n/5)^5$. Note that when $n$ is divisible by 5, the balanced blow-up of a $5$-cycle shows that $\ex(n,C_5,C_3)\geq (n/5)^5$. Here and below, a balanced blow-up of a graph $F$ is obtained by replacing the vertices of $F$ with independent sets of size differing by at most one from each other and replacing the edges of $F$ by complete bipartite graphs between the corresponding independent sets.\footnote{Note that when $n$ is not divisible by $|V(F)|$, the balanced blow-up of $F$ on $n$ vertices is not necessarily unique.} In 1989, Gy\H ori \cite{Gyo89} proved that $\ex(n,C_5,C_3)\leq 1.03(n/5)^5$. Erd\H os's conjecture was proved, using the method of flag algebras, by Grzesik \cite{Grz12} and independently by Hatami, Hladk\'y, Kr\'al', Norine and Razborov \cite{HHKNR13}.

The case $T=C_3$, $H=C_{2k+1}$ has also been studied. Bollob\'as and Gy\H ori \cite{BGy08} showed that $(1+o(1))\frac{n^{3/2}}{3\sqrt{3}}\leq \ex(n,C_3,C_5)\leq (1+o(1))\frac{5}{4}n^{3/2}$. Gy\H ori and Li \cite{GyL12} obtained results bounding $\ex(n,C_3,C_{2k+1})$ which were improved by Alon and Shikhelman \cite{AS16} and F\"uredi and \"Ozkahya~\cite{FO17}. Gishboliner and Shapira \cite{GS20} determined the order of magnitude of $\ex(n,C_k,C_{\ell})$ for all $k$ and $\ell$, while this was done independently for all even $k$ and $\ell$ by Gerbner, Gy\H ori, Methuku and Vizer \cite{GGyMV20}.

Our main focus in this paper is also the case where $T$ and $H$ are odd cycles. The function $\ex(n,C_{k},C_{\ell})$ exhibits very different behaviour depending on whether $k<\ell$ or $k>\ell$. Indeed, in general it is well-known that when there is a graph homomorphism from $H$ to $T$, then $\ex(n,T,H)=o(n^{|V(T)|})$ (called the \emph{sparse} case), whereas if there is no such homomorphism, then $\ex(n,T,H)=\Theta(n^{|V(T)|})$ (called the \emph{dense} case). In this paper, we will be concerned with the dense case, which corresponds to $\ex(n,C_k,C_{\ell})$ with $k>\ell$ (assuming that both $k$ and $\ell$ are odd).

Note that the conjecture of Erd\H os about $\ex(n,C_5,C_3)$ falls into this case. Extending this result to longer cycles, Grzesik and Kielak \cite{GK22} proved that for each odd $k\geq 7$, every $n$-vertex graph with no odd cycle of length less than $k$ contains at most $(n/k)^k$ cycles of length~$k$, with equality attained only by the balanced blow-up of a $k$-cycle (and only for $n$ divisible by $k$). As a corollary of this result, Grzesik and Kielak showed that for any odd $k\geq 7$, we have $\ex(n,C_{k},C_{k-2})= (n/k)^k+o(n^k)$. Our main result determines the function $\ex(n,C_k,C_{k-2})$ exactly and characterises the extremal constructions (for large $n$). Along the way, we also obtain a stability result (see Lemma \ref{lemma:G}).

\begin{restatable}{theorem}{mainthm}
\label{thm:k,k-2}
For each odd integer $k \geq 7$, there exists $n_0$ such that if $n\geq n_0$ and $G$ is a $C_{k-2}$-free graph on $n$ vertices maximising the number of $k$-cycles, then $G$ is a balanced blow-up of $C_k$. In particular, for each $n\geq n_0$, we have $\ex(n,C_k,C_{k-2})\leq (n/k)^k$.
\end{restatable}

We remark that, similarly to the argument of Grzesik and Kielak, our proof does not use flag algebras.

Note that all balanced blow-ups of $C_k$ on $n$ vertices contain the same number of $k$-cycles, so Theorem \ref{thm:k,k-2} shows that they are precisely the extremal graphs. In the concluding remarks section, we show that when $n$ is not assumed to be sufficiently large, then the extremal construction can be different from a balanced blow-up of a $k$-cycle.

Grzesik and Kielak conjectured that more generally, for every odd $k$ and $\ell$ with $k>\ell$, the balanced blow-up of a $C_{\ell+2}$ asymptotically attains the maximum number of $k$-cycles that an $n$-vertex $C_{\ell}$-free graph can contain.\footnote{Note that although the balanced blow-up is not necessarily unique, all balanced blow-ups contain asymptotically the same number of $k$-cycles.} We disprove this conjecture, showing that when $k$ is sufficiently large compared to $\ell$, there is an unbalanced blow-up of $C_{\ell+2}$ which contains asymptotically more copies of $C_k$ than the balanced blow-up.

\begin{theorem} \label{thm:k,l}
    For each odd $\ell\geq 3$, there is some $k_0$ such that whenever $k\geq k_0$ is odd, then there is a sequence of $n$-vertex $C_{\ell}$-free graphs which contain asymptotically more copies of $C_k$ than the balanced blow-up of $C_{\ell+2}$ on $n$ vertices.
\end{theorem}

In fact, as we will demonstrate in the concluding remarks section, we may choose $k_0=\ell+C\frac{\ell}{\log \ell}$ for some absolute constant $C$.

The rest of this paper is organized as follows. In Section \ref{sec:proofs}, we prove Theorem~\ref{thm:k,k-2} and Theorem~\ref{thm:k,l}. In Section \ref{sec:remarks}, we give some concluding remarks.

\section{The proof of Theorem \ref{thm:k,k-2} and Theorem \ref{thm:k,l}} \label{sec:proofs}

\subsection{The proof of Theorem \ref{thm:k,k-2}}

The proof of Theorem \ref{thm:k,k-2} uses a so-called \emph{stability} argument and is structured as follows. In Lemma \ref{lemma:rmvodd} we prove that any $C_{k-2}$-free graph $G$ contains a subgraph $G'$ that is $C_\ell$-free for all $3\leq \ell <k$ odd and has nearly as many $k$-cycles as $G$. In Lemma \ref{lemma:dense} and Lemma~\ref{lemma:mindeg} we show that if $G'$ has nearly $(n/k)^k$ $k$-cycles, then it must have a structure close to a blow-up of a $C_k$. In Lemma \ref{lemma:G} we obtain a similar result for the original graph $G$, which is then used to complete the proof of Theorem \ref{thm:k,k-2}.

\textbf{Notation.} For disjoint vertex subsets $S,T$ of a graph $G$, let $e_{G}(S,T)$ denote the number of edges of $G$ which have one endpoint in $S$ and one endpoint in $T$. Let $\rho_{G}(S,T)=\frac{e_{G}(S,T)}{|S||T|}$ denote the fraction of pairs in $S\times T$ that are edges of $G$. For a vertex $v$ and a vertex set $T$ (which may contain $v$), we write $e_G(v,T)$ for the number of neighbours of $v$ in $T$. Furthermore, we let $\rho_G(v,T)=e_G(v,T)/|T|$. For simplicity of notation, we drop the subscript if the graph $G$ is clear from context. For any vertices $v$ and $w$, by $\textrm{dist}(v, w)$ we denote the distance between the vertices $v$ and $w$ in $G$.
Also, for $v \in V (G)$, write $N(v)$ for the neighbourhood of a vertex $v$. Logarithms are to base $e$.

In order to delete short odd cycles we will use the celebrated Graph Removal Lemma (see, for example, the survey of Conlon and Fox \cite{CF13}).

\begin{lemma}[Graph Removal Lemma] \label{lemma:grl}
    For any graph $H$ and any $\epsilon_0>0$, there exists $\delta_0>0$ such that any graph on $n$ vertices which contains at most $\delta_0 n^{v(H)}$ copies of $H$ may be made $H$-free by removing at most $\epsilon_0 n^2$ edges.
\end{lemma}

We will also need the following easy bound on the generalized Turán number of odd cycles.

\begin{lemma} \label{prop:l,k-2}
    For any odd numbers $3\leq \ell < k$ we have $\ex(n,C_{\ell},C_{k})=o(n^{\ell})$.
\end{lemma}

Gishboliner and Shapira \cite{GS20} showed the stronger result that there is an absolute constant $c$, such that for $1\leq a <b$
$$\ex(n,C_{2a+1},C_{2b+1})\leq \begin{cases} cb^2n^{1+1/b}& \text{if } a=1\\
c(2a+1)^{2a}(2b +1)^{a+1}n^a &\text{if } a\geq 2.\end{cases}$$

We give a short proof of Lemma \ref{prop:l,k-2} for completeness.
\begin{proof}
    Let $G$ be a $C_k$-free graph on $n$ vertices. Let $P=vv_1v_2\dots v_{\ell-4}u$ be a path of length $\ell-3$ (note that if $\ell=3$, then $v=u$), let $A=(N(u)\setminus N(v))\setminus V(P)$, $B=(N(v)\setminus N(u))\setminus V(P)$ and $C=(N(u)\cap N(v))\setminus V(P)$. Now observe that the number of $\ell$-cycles containing $P$ is at most $e(A,B)+e(A,C)+e(B,C)+2e(C)$. If any of the bipartite graphs defined by the edges of $(A,B)$, $(A,C)$ or $(B,C)$, or the graph defined by the edges of $C$ contains a path of length $k-\ell+1$, then this path and $P$ form a cycle of length $k$, which is a contradiction. A graph on $m$ vertices without a path of length $k-\ell+1$ can contain at most $\frac{k-\ell}{2}m$ edges (see \cite{EG59}), so the number of $k$-cycles containing $P$ is at most
    $$e(A,B)+e(A,C)+e(B,C)+2e(C)\leq \frac{k-\ell}{2}(|A|+|B|+|A|+|C|+|B|+|C|+2|C|)\leq 2(k-\ell)n.$$
    The number of paths of length $\ell-3$ is at most $n^{\ell-2}$, hence $\ex(n,C_\ell,C_k)\leq 2(k-\ell)n^{\ell-1}$.\end{proof}

\begin{lemma} \label{lemma:rmvodd}
    For each odd integer $k\geq 7$ and $\delta>0$ there exists $n_0$ such that if $G$ is a $C_{k-2}$-free graph on $n\geq n_0$ vertices, then it is possible to delete edges from $G$ to obtain a graph $G'$ that is $C_\ell$-free for all odd $3\leq \ell <k$, such that the number of $k$-cycles in $G'$ and the number of $k$-cycles in $G$ differ by at most $\delta (n/k)^k$.
\end{lemma}

\begin{proof}
    Let $\epsilon_0=\delta/k^{k+1}$. By the Graph Removal Lemma, there exists $\delta_0>0$ such that, for every $3\leq \ell\leq k-1$, any $n$-vertex graph with at most $\delta_0 n^{\ell}$ copies of $C_{\ell}$ may be made $C_{\ell}$-free by removing at most $\epsilon_0 n^2$ edges. Choose $n_0$ such that for all $n\geq n_0$ we have $\ex(n,C_{\ell},C_{k-2})\leq \delta_0 n^{\ell}$ for all odd numbers $3\leq \ell < k$; this is possible by Lemma \ref{prop:l,k-2}. Let $G$ be a $C_{k-2}$-free graph on $n\geq n_0$ vertices. Now for a given odd number $3\leq \ell <k-2$, $G$ may be made $C_{\ell}$-free by removing at most $\epsilon_0 n^2$ edges, hence $G$ may be made $C_\ell$-free for all odd $3\leq \ell <k-2$ by removing at most $(k-3)\cdot \epsilon_0 n^2<(\delta/k^k) \cdot n^2$ edges. As any given edge of $G$ is contained in at most $n^{k-2}$ $k$-cycles, the number of $k$-cycles in the new graph and the number of $k$-cycles in $G$ differ by at most $\delta (n/k)^k$.
\end{proof}

The following definitions, Lemma \ref{lem:weightsandkcycles} and Lemma \ref{claim:GK} closely follow the proof of the main result of Grzesik and Kielak \cite{GK22}.

Fix an odd integer $k\ge 7$ and let $G'$ be any $n$-vertex graph without $C_{\ell}$ for all odd $\ell$ between $3$ and $k-2$. Since there are no odd cycles of length smaller than $k$, each $k$-cycle in $G'$ is induced.

Call a sequence $D = (z_i)_{i=0}^{k-1}$ of vertices of $G'$ \textit{good} if $z_0z_1z_3z_2z_4\dots z_{k-1}$ is a cycle (so $z_2$ and $z_3$ are in reverse order). For a fixed good sequence $D$, we define the following sets:
\begin{align*}
    A_0(D)&=V(G'),\\
    A_1(D)&=N(z_0),\\
    A_2(D)&=\{w\not\in N(z_0):\textrm{dist}(z_1,w)=2\},\\
    A_3(D)&=N(z_1)\cap N(z_2),\\
    A_4(D)&=\{w:z_0z_1z_3z_2w \text{ is an induced path}\},\\
    A_i(D)&=\{w:z_0z_1z_3z_2z_4\dots z_{i-1}w \text{ is an induced path}\}\text{ for }5\le i\le k-2,\\
    A_{k-1}(D)&=\{w:z_0z_1z_3z_2z_4\dots z_{k-2}w \text{ is an induced cycle}\}.
\end{align*}

Define the weight $w(D)$ of a good sequence $D$ as
\begin{align*}
    w(D)=\prod_{i=0}^{k-1}|A_i(D)|^{-1}=\frac{1}{n}\prod_{i=1}^{k-1}|A_i(D)|^{-1}.
\end{align*}

Note that the set $A_i(D)$ only depends on $z_0,z_1,\dots, z_{i-1}$, so write $A_i(D)=A_i(z_0,z_1,\dots,z_{i-1})$. Consider the following process of sampling a sequence of vertices $(w_i)_{i=0}^{k-1}$. Start with choosing $w_0$ randomly from $V(G')=A_0$. If we have already chosen $w_0,\dots,w_{i-1}$, choose $w_i$ randomly from $A_i(w_0,\dots,w_{i-1})$ (if this set is empty, terminate the process). Then, $w(D)$ is just the probability that the sequence $(w_i)_{i=0}^{k-1}$ obtained in this random process is equal to $D$. In particular, the sum of weights of all good sequences is at most one, since it is the sum of probabilities of pairwise disjoint events.

Fix a $k$-cycle $v_0v_1\dots v_{k-1}$ in $G'$, let $C=\{v_0,v_1,\dots,v_{k-1}\}$ be the set of its vertices and, for $0\leq j\leq k-1$, let $D_j=(v_j,v_{j+1},v_{j+3},v_{j+2},v_{j+4},\dots,v_{j+k-1})$,  where the indices are considered modulo $k$. Note that the sets $D_j$ are the good sequences with the same orientation corresponding to this cycle (they are half the total number of good sequences corresponding to this cycle; the other half corresponds to the good sequences with the reverse orientation).

The following lemma, which is implicit in \cite{GK22}, relates the weights of good sequences to the number of $k$-cycles.

\begin{lemma} \label{lem:weightsandkcycles}
    Let $G'$ be and $D_j$ be defined as above. Then the number of $k$-cycles in $G'$ is at most the maximum of $\left(2\sum_{j=0}^{k-1}w(D_j)\right)^{-1}$ over all $k$-cycles in $G'$.
\end{lemma}

\begin{proof}
Let $N$ be the maximum of $(2\sum_{j=0}^{k-1}w(D_j))^{-1}$ over all $k$-cycles in $G'$. Then $\sum_{j=0}^{k-1}w(D_j)\geq (2N)^{-1}$ holds for any $k$-cycle in $G'$. Summing this over all $k$-cycles (with both orientations) we get an inequality in which the left hand side is upper bounded by one (since the sum of the weights of all good sequences is at most one), while the right hand side becomes $N^{-1}$ times the number of $k$-cycles in $G'$. Hence, we conclude that the total number of $k$-cycles is upper bounded by $N$.
\end{proof}

Still for a fixed $k$-cycle in $G'$, let $n_{i,j}=|A_i(D_j)|$. Using the inequality between the harmonic mean and the geometric mean of $k$ terms
and the inequality between the geometric mean and the arithmetic mean of $k(k - 1)$ terms, we obtain
\begin{align}
    \left(2\sum_{j=0}^{k-1}w(D_j)\right)^{-1}&=\left(2\sum_{j=0}^{k-1}\prod_{i=0}^{k-1}n_{i,j}^{-1}\right)^{-1} \nonumber \\
    &=n\left(\sum_{j=0}^{k-1}\left(\frac{n_{1,j}}{2}\right)^{-1}\prod_{i=2}^{k-1}n_{i,j}^{-1}\right)^{-1} \nonumber \\
    &\leq \frac{n}{k}\left(\prod_{j=0}^{k-1}\frac{n_{1,j}}{2}\prod_{i=2}^{k-1}n_{i,j}\right)^{\frac{1}{k}} \nonumber \\
    &\leq \frac{n}{k}\left(\frac{1}{k(k-1)}\sum_{j=0}^{k-1}\left(\frac{n_{1,j}}{2}+\sum_{i=2}^{k-1}n_{i,j}\right)\right)^{k-1}. \label{eqn:HMAMGM}
\end{align}

Notice that any vertex $w \in V(G')$ has at most 2 neighbours in $C$, otherwise $G'$ has a shorter odd cycle. The following definition plays a key role in our proof.

\begin{definition}
    Let $U(C)=\{w\in V(G'):|N(w)\cap C|=2\}$, and let $M_{G'}$ be the maximum of $|U(C)|/n$ over all $k$-cycles in $G'$.
\end{definition}

\begin{lemma} \label{claim:GK}
    The following inequality holds for any $k$-cycle in $G'$ (with $n_{i,j}$ defined as above):
    \begin{align*}
    \sum_{j=0}^{k-1}\left(\frac{n_{1,j}}{2}+\sum_{i=2}^{k-1}n_{i,j}\right)\leq n(k-2+M_{G'}).
\end{align*}
\end{lemma}

\begin{proof} It is enough to prove that the contribution to the above sum of any vertex $w\in U(C)$ is at most $k-1$ and the contribution of any vertex $w\not\in U(C)$ is at most $k-2$ (here we say that a vertex $w$ contributes to $n_{i,j}$ if $w\in A_i(D_j)$).

Since $G'$ contains no odd cycle of length less than $k$, any vertex $w\in V(G')$ has at most 2 neighbours in $C$, and if $w$ has two neighbours in $C$, then those neighbours must have distance two on the cycle. Furthermore, again since $G'$ has no odd cycle of length less than $k$, each vertex $w$ satisfies the following property:

($\star$) There are at most three vertices in $C$ at distance exactly 2 from $w$, and any two such vertices are not adjacent.

To see that ($\star$) holds, note that if there are two adjacent vertices which are of distance two from $w$, then we obtain a circuit of length $5$, which must contain an odd cycle of length at most $5$. Moreover, assume for the sake of contradiction that there are at least $4$ vertices in $C$ of distance two from $w$, let us call them $x_1,x_2,x_3,x_4$ in their clockwise order on the cycle. Let $P_i$ be a path of length two from $w$ to $x_i$ and let $Q_i$ be the clockwise arc from $x_i$ to $x_{i+1}$ on the cycle (with indices considered modulo $4$). Each $Q_i$ has length at least $2$, so (as the sum of the lengths of the four arcs is $k$) each arc $Q_i$ has length at most $k-6$. Furthermore, at least one of them has odd length, say $Q_j$ does. Then the circuit consisting of $P_j$, $Q_j$ and $P_{j+1}$ has odd length at most $k-2$, so it must contain an odd cycle of length at most $k-2$, which is a contradiction.

Now let us see how much a vertex can contribute to the sum $\sum_{j=0}^{k-1}\left(\frac{n_{1,j}}{2}+\sum_{i=2}^{k-1}n_{i,j}\right)$.

If $w$ has no neighbours in $C$, then, for each $j$, it can contribute only to
$n_{2,j}$. Moreover, if for some $j$ we have $\textrm{dist}(w, v_j ) = 2$, then $\textrm{dist}(w, v_{j-1}) > 2$ and $\textrm{dist}(w, v_{j+1}) > 2$ by ($\star$), and so $w$ does not contribute to $n_{2,j-2}$ and $n_{2,j}$. Therefore, such $w$ contributes in total by at most $k-2$.

Assume, now, that $w$ has exactly one neighbour in $C$ - by symmetry, we may assume that this neighbour is $v_0$. Because of having only one neighbour, $w$ does not contribute
to $n_{3,j}$ or to $n_{k-1,j}$ for any $j$. In order to contribute to $n_{i,j}$ for $i \not\in \{2, 3, k -1\}$, $w$ needs
to be adjacent to $v_{i+j-1}$, and so it can contribute only to $n_{1,0}$ and $n_{i,k-i+1}$ for $4\leq i\leq k-2$. Finally, $w$ can contribute to $n_{2,j}$ only if $\textrm{dist}(w, v_{j+1}) = 2$ and $w\not\in N(v_j )$. By ($\star$), there are at most three vertices in $C$ at distance $2$ from $w$, but one of them is $v_1$ and $w\in N(v_0)$, so $w$ contributes to $\sum_{j=0}^{k-1}n_{2,j}$ by at most $2$. It follows that in this case $w$ contributes to the considered sum in total by at most $k - 3 + \frac{1}{2}\leq k-2$.

Finally, assume that $w$ has exactly two neighbours in $C$. As noted above, these neighbours
have to be at distance 2 in $C$. By symmetry, we may assume that they are $v_{k-1}$ and $v_1$. Then, $\textrm{dist}(w, v_i) = 2$ for $i \in \{ k - 2, 0, 2 \}$ (or $w=v_0$), and there are no more $i$ with this property by ($\star$). Therefore, $w$ contributes only to $n_{1,k-1}, n_{1,1}, n_{2,k-3}, n_{3,k-2}$, and $n_{i,k-i}$ for $4\leq i\leq k-1$, hence $w$ contributes to the considered sum in total by $k-1$.
\end{proof}

The upshot of the above is the following result.

\begin{lemma} \label{lem:relate cycles and M}
    Let $k\geq 7$ be odd and let $G'$ be a graph on $n$ vertices that is $C_\ell$-free for all odd $3\leq \ell<k$. Then the number of $k$-cycles in $G'$ is at most  $(1-\frac{1-M_{G'}}{k-1})^{k-1}\frac{n^k}{k^k}$.
\end{lemma}

\begin{proof}
Using Lemma \ref{lem:weightsandkcycles}, equation (\ref{eqn:HMAMGM}) and Lemma \ref{claim:GK}, the number of $k$-cycles in $G'$ is at most
\begin{align*}
    \frac{n}{k}\left(\frac{1}{k(k-1)} n(k-2+M_{G'})\right)^{k-1}=\left(1-\frac{1-M_{G'}}{k-1}\right)^{k-1}\frac{n^k}{k^k},
\end{align*}
as desired.
\end{proof}

Our next result gives very strong structural information about $\{C_3,C_5,\dots,C_{k-2}\}$-free $n$-vertex graphs with nearly $(n/k)^k$ many $k$-cycles. It can be viewed as a stability version of the main result of Grzesik and Kielak \cite{GK22}.

\begin{lemma}\label{lemma:dense}
For every odd $k\geq 7$, $\epsilon_1,\epsilon_2>0$ there exists $\delta>0$ such that if $G'$ is a graph on $n$ vertices that is $C_\ell$-free for all odd $3\leq \ell<k$, and $G'$ has at least $(1-\delta)(n/k)^k$ $k$-cycles, then $G'$ contains a set $U$ of vertices with $|U|\geq (1-\epsilon_1)n$, and a partition into independent sets $U=U_0\cup U_1\cup \dots\cup U_{k-1}$ such that for all $0\leq i<j\leq k-1$, 

 \begin{itemize}
     \item if $i-j\equiv \pm 1\pmod{k}$, then $\rho(U_i,U_j)\geq 1-\epsilon_2$,
     \item otherwise $\rho(U_i,U_j)=0$.
 \end{itemize}
\end{lemma}
\begin{proof} Decrease $\epsilon_1$ if necessary such that $0<\epsilon_2-\epsilon_1k^k$ and take $\delta>0$ with $1-\delta>\left(1-\epsilon_1/(k-1)\right)^{k-1}$ and $1-\delta>1-(\epsilon_2-\epsilon_1k^k)$.

Then by considering the number of $k$-cycles in $G'$, we obtain, using Lemma \ref{lem:relate cycles and M}, that
\begin{align*}
    \left(1-\frac{\epsilon_1}{k-1}\right)^{k-1}\frac{n^k}{k^k}< (1-\delta)\frac{n^k}{k^k}\leq \left(1-\frac{1-M_{G'}}{k-1}\right)^{k-1}\frac{n^k}{k^k}.
\end{align*}
Hence $M_{G'}>1-\epsilon_1$, so there exists a $k$-cycle $v_0v_1\dots v_{k-1}$ such that for $C=\{v_0,v_1,\dots,v_{k-1}\}$ we have $|U(C)|>(1-\epsilon_1)n$. Let $U=U(C)$. Recall that if some $w\in V(G')$ has 2 neighbours in $C$, then they must be at distance 2 on the cycle $v_0v_1\dots v_{k-1}$. Hence $U$ is the disjoint union of the sets $U_i=\{w\in V(G'):N(w)\cap C=\{v_{i-1},v_{i+1}\}\}$ for $0\leq i\leq k-1$.

If for some $i$ there is an edge $xy$ in $U_i$, then $xyv_{i-1}$ is a triangle, which is a contradiction. If there is an edge $xy$ from $U_i$ to $U_{i+j}$ with $j\not\in\{0,1,k-1\}$, then $xv_{i+1}v_{i+2}\dots v_{i+j-1}y$ is a $(j+1)$-cycle and $xv_{i-1}v_{i-2}\dots v_{i+j+1}y$ is a $(k-j+1)$-cycle. As both $j+1$ and $k-j+1$ are less than $k$ and one of them is odd, this is a contradiction.

 Suppose for the sake of contradiction that for some $0\leq i\leq k-1$, $\rho(U_i,U_{i+1})\leq 1-\epsilon_2$. Note that each $k$-cycle contained in $U$ must have an edge from $E(U_i,U_{i+1})$ and a vertex from each $U_j$ for $j\not\in\{ i,i+1\}$, and this edge and the $k-2$ vertices uniquely determine the $k$-cycle. Therefore there are at most $(1-\epsilon_2)\prod_{i=0}^{k-1}|U_i|$ $k$-cycles contained in $U$. This is at most $(1-\epsilon_2)(n/k)^k$ by AM-GM. The number of $k$-cycles in $G'$ that are not contained in $U$ is bounded by $\epsilon_1 n^k$, since there are at most $\epsilon_1 n$ vertices in $G'$ which are not in $U$. Therefore, considering the total number of $k$-cycles in $G'$, we obtain
 \begin{align*}
     (1-\delta)\frac{n^k}{k^k}\le (1-\epsilon_2)\frac{n^k}{k^k}+\epsilon_1 n^k=\left(1-(\epsilon_2-\epsilon_1k^k)\right)\frac{n^k}{k^k},
 \end{align*}
contradicting the assumptions about $\delta$. Hence the sets $U_0,U_1,\dots,U_{k-1}$ satisfy the conditions stated in the lemma. \end{proof}

We want to use the high edge density between the sets $U_i$ and $U_{i+1}$ to find large subsets $W_i\subset U_i$ which have large minimal degree to $W_{i\pm 1}$.

\begin{lemma}\label{lemma:mindeg}
For every odd $k\geq 7$, $\epsilon_3,\epsilon_4>0$ there exists $\delta>0$ such that if $G'$ is a graph on $n$ vertices that is $C_\ell$-free for all odd $3\leq \ell<k$, and $G'$ has at least $(1-\delta)(n/k)^k$ $k$-cycles, then $G'$ contains a set $W$ of vertices with $|W|\geq (1-\epsilon_3)n$, and a partition into independent sets $W=W_0\cup W_1\cup \dots\cup W_{k-1}$, such that  $|W_i|\geq n/(6k)$ and for all $0\leq i,j\leq k-1$ and $v\in W_i$,
 \begin{itemize}
     \item if $i-j\equiv \pm 1\pmod{k}$, then $\rho(v,W_j)\geq 1-\epsilon_4$,
     \item otherwise $\rho(v,W_j)=0$.
 \end{itemize}
\end{lemma}

\begin{proof} We will use the following claim.
\begin{claim}\label{claim:bip}
If $H=X\cup Y$ is a bipartite graph and $\rho(X,Y)>1-\epsilon_0^2$, then there exists $X'\subset X$ such that $|X'|>(1-\epsilon_0)|X|$ and for all $v\in X'$, $\rho(v,Y)>1-\epsilon_0$.
\end{claim}
\begin{proof} Let $X'=\{v\in X:\rho(v,Y)>1-\epsilon_0\}$. Then
\begin{align*}
    \epsilon_0(|X|-|X'|)&=\sum_{x\in X\setminus X'}\epsilon_0\leq \sum_{x\in X\setminus X'}(1-\rho(x,Y))\\
    &\leq\sum_{x\in X}(1-\rho(x,Y))=|X|(1-\rho(X,Y))\\
    &<|X|\epsilon_0^2.
\end{align*}
Therefore $|X'|>(1-\epsilon_0)|X|$. \end{proof}

Decrease $\epsilon_3$ if necessary such that $\epsilon_3<1/(2k^k)$. Choose $\epsilon_0,\epsilon_1>0$ such that $(1-2\epsilon_0
)(1-\epsilon_1)>1-\epsilon_3$ and $1-3\epsilon_0>1-\epsilon_4$.

Let $\delta>0$, $U_0,U_1\dots,U_{k-1}$ be given by Lemma \ref{lemma:dense} for $\epsilon_1$ and $\epsilon_2=\epsilon_0^2$, and decrease $\delta$ if necessary such that $\epsilon_3<\left(1/2-\delta\right)/k^k$ (this is possible since $\epsilon_3<1/(2k^k)$).

For all $0\leq i\leq k-1$ use Claim \ref{claim:bip} with $X=U_i$, $Y=U_{i\pm 1}$ to obtain $V_i^{\pm}\subset U_i$ such that $|V_i^{\pm}|> (1-\epsilon_0)|U_{i}|$ and for all $v\in V_i^{\pm}$, $\rho(v,U_{i\pm 1})> 1-\epsilon_0$.

Let $W_i=V_i^+\cap V_i^-$ and $W=\cup_{i=0}^{k-1}W_i$. Then $|W_i|\geq (1-2\epsilon_0)|U_i|$, so $\left|W\right|\geq (1-2\epsilon_0)(1-\epsilon_1)n>(1-\epsilon_3)n$. Also for all $v\in W_i$ we have 
\begin{align*}
    \rho(v,W_{i\pm1})= \frac{e(v,W_{i\pm 1})}{|W_{i\pm 1}|}\geq \frac{e(v,U_{i\pm1})-|U_{i\pm 1}\setminus W_{i\pm 1}|}{|U_{i\pm 1}|}>(1-\epsilon_0)-2\epsilon_0=1-3\epsilon_0>1-\epsilon_4.
\end{align*}

As $W_i\subset U_i$ for all $0\leq i\leq k-1$, we still have $\rho(W_i,W_j)=0$ for $i-j\not\equiv \pm 1 \pmod{k}$.
If $|W_i|<n/(6k)$ for some $i$, then by AM-GM the number of $k$-cycles contained in $W$ is at most
\begin{align*}
    \prod_{i=0}^{k-1}|W_i|\leq\frac{n}{6k}\frac{n^{k-1}}{(k-1)^{k-1}}=\frac{1}{6}\left(1+\frac{1}{k-1}\right)^{k-1}\frac{n^k}{k^k}<\frac{1}{2}\frac{n^k}{k^k}.
\end{align*} The number of $k$-cycles which are not contained in $W$ is at most $\epsilon_3n^k<\left(1/2-\delta\right)(n/k)^k$, which contradicts that the number of $k$-cycles in $G'$ is at least $(1-\delta)(n/k)^k$. 

Therefore the sets $W_0,W_1,\dots,W_{k-1}$ satisfy the conditions stated in the lemma. \end{proof}

In the next lemma, we combine the results of Lemma \ref{lemma:rmvodd} and Lemma \ref{lemma:mindeg} to prove that a $C_{k-2}$-free graph which contains many copies of $C_k$ must have a structure close to a blow-up of a $k$-cycle.

\begin{lemma} \label{lemma:G} For every odd $k\geq 7$ and $\epsilon_5>0$, there exist $\delta >0$ and $n_0>0$ such that if $G$ is a $C_{k-2}$-free graph on $n\geq n_0$ vertices and $G$ has at least $(1-\delta)(n/k)^k$ $k$-cycles, then $V(G)$ contains a set $B$ with $|B|\geq (1-\epsilon_5)n$, and a partition into independent sets $B=B_0\cup B_1\cup \dots \cup B_{k-1}$ with $|B_i|\geq n/(6k)$, such that the induced subgraph on $B$ is a subgraph of the blow-up of $C_k$ obtained by replacing the vertices of $C_k$ with independent sets $B_0,B_1,\dots ,B_{k-1}$. Furthermore, if $v\in V(G)\setminus B$, then there exists $0\leq j\leq k-1$ such that 
\begin{itemize}
     \item for all $0\leq i\leq k-1$, if $i\not\equiv j\pm 1\pmod{k}$, then $\rho(v,B_i)=0$,
     \item $\rho(v,B_{j-1})\leq 3/4$ or $\rho(v,B_{j+1})\leq 3/4$.
\end{itemize}\end{lemma}

\begin{proof} Take any $0<\epsilon_4<1/3$. Choose $\epsilon_3>0$ such that $\epsilon_3<\epsilon_5$, $C:=(1-6k\epsilon_3)\cdot 2/3>1/2$ and $2/3+6k\epsilon_3\leq 3/4$. Apply Lemma \ref{lemma:mindeg} with $\epsilon_3,\epsilon_4$ and let $\delta_0$ be the value of $\delta$ provided by that lemma. Take $\delta=\delta_0/2$. Let $n_0$ be large enough such that Lemma \ref{lemma:rmvodd} applies for $\delta$, and such that $n_0>6k^2/(2C-1)$. Let $G$ be a $C_{k-2}$-free graph on $n\geq n_0$ vertices with at least $(1-\delta)(n/k)^k$ $k$-cycles. Apply Lemma \ref{lemma:rmvodd} to $G$ and $\delta$ to get a $\{C_3,C_5,\dots,C_{k-2}\}$-free graph $G'$ that has at least $(1-\delta)(n/k)^k-\delta(n/k)^k=(1-\delta_0)(n/k)^k$ $k$-cycles. Apply Lemma \ref{lemma:mindeg} to $G'$ to obtain sets $W=W_0\dot\cup W_1\dot\cup \dots \dot\cup W_{k-1}$ such that $|W|\geq (1-\epsilon_3)n$, and for all $0\leq i\leq k-1$ we have $|W_i|\geq n/(6k)$ and if $v\in W_i$ and $j\equiv i\pm 1\pmod{k}$, then $\rho_G(v,W_j)\geq \rho_{G'}(v,W_j)\geq 1-\epsilon_4$.

For all $0\leq i\leq k-1$ let $W_i'=\{v\in V(G):|N_G(v)\cap W_{i+1}|>\frac{2}{3}|W_{i+1}| \text{ and } |N_G(v)\cap W_{i-1}|>\frac{2}{3}|W_{i-1}|\}$. Observe that $W_i\subset W_i'$. Choose disjoint sets $B_0,B_1,\dots,B_{k-1}$ with $W_i\subset B_i\subset W_i'$ and $B=\cup_{i=0}^{k-1}B_i=\cup_{i=0}^{k-1}W_i'$. Now clearly $|B_i|\geq |W_i|\geq n/(6k)$. Using that $|B_i|\leq |W_i|+\epsilon_3n$, we also get $|W_i|\geq (1-6k\epsilon_3)|B_i|$. So if $v\in B_i$, then
\begin{align*}
    e_G(v,B_{i\pm 1})&\geq e_G(v,W_{i\pm 1})\geq \frac{2}{3}|W_{i\pm 1}|\geq \frac{2}{3}\left(1-6k\epsilon_3\right)|B_{i\pm 1}|=C|B_{i\pm 1}|.
\end{align*}

\begin{claim}\label{claim:nbr} For all $v\in V(G)$, if $v$ has neighbours in $B_i$ and $B_{i+j}$ with $0< j<\frac{k}{2}$, then $j=2$. \end{claim}

\begin{proof} Suppose that $j\ne 2$. By symmetry we may assume that $i=0$. Let $v_0$ be a neighbour of $v$ in $B_0$ and let $v_{k-4}$ be a neighbour of $v$ in $B_j$. We will construct vertices $v_1, v_2,\dots ,v_{k-5}$ in $B$ one by one such that $vv_0v_1\dots v_{k-5}v_{k-4}$ forms a $(k-2)$-cycle, hence get a contradiction. Consider two cases based on the parity of $j$.

If $j$ is odd, we look for $v_i$ with $v_{i-1}v_i\in E(G)$ and
\begin{itemize}
    \item $v_i\in B_i$ if $1\leq i\leq j$,
    \item $v_i\in B_j$ if $j<i<k-5$ and $i$ is odd,
    \item $v_i\in B_{j-1}$ if $j<i<k-5$ and $i$ is even.
\end{itemize}

This is possible by induction, since if $v_{i-1}\in B_{a}$ then
\begin{align*}
    e_G(v_{i-1},B_{a\pm 1}\setminus \{v,v_0,\dots,v_{i-1},v_{k-4}\})\geq e_G(v_{i-1},B_{a\pm 1})-k\geq C|B_{a\pm 1}|-k\geq \frac{Cn}{6k}-k>0.
\end{align*}
Since $j<k/2$, we have $k-6\geq j-2$, so $v_{k-6}\in B_j\cup B_{j-2}$. Combining this with $v_{k-4}\in B_j$, we obtain
 \begin{align*}
     \left|\left(N_G(v_{k-4})\cap N_G(v_{k-6})\cap B_{j-1}\right)\setminus \{v,v_0,\dots,v_{k-2}\} \right|\geq (2C-1)|B_{j-1}|-k\geq \frac{(2C-1)n}{6k}-k>0,
 \end{align*}
 hence $v_{k-5}\in B_{j-1}$ can be chosen such that $vv_0v_1\dots v_{k-5}v_{k-4}$ forms a $(k-2)$-cycle.

If $j$ is even (and hence $j\geq 4$ by the assumption that $j\neq 2$), we look for $v_i$ with $v_{i-1}v_i\in E(G)$ and
\begin{itemize}
    \item $v_i\in B_{k-i}$ if $1\leq i\leq k-j-2$,
    \item $v_i\in B_{j+2}$ if $k-j-2<i<k-5$ and $i$ is odd,
    \item $v_i\in B_{j+1}$ if $k-j-2<i<k-5$ and $i$ is even.
\end{itemize}

This is possible by induction and the same inequality as above. Now, using $j\geq 4$, we have $v_{k-6}\in B_{j+2}$. Combining this with $v_{k-4}\in B_j$, we obtain
 \begin{align*}
     \left|\left(N_G(v_{k-4})\cap N_G(v_{k-6})\cap B_{j+1}\right)\setminus \{v,v_0,\dots,v_{k-2}\} \right|\geq (2C-1)|B_{j+1}|-k\geq \frac{(2C-1)n}{6k}-k>0,
 \end{align*}
 so $v_{k-5}\in B_{j+1}$ can be chosen such that $vv_0v_1\dots v_{k-5}v_{k-4}$ forms a $(k-2)$-cycle.\end{proof}

Since each $v\in B_i$ has a neighbour in both $B_{i-1}$ and $B_{i+1}$, from Claim \ref{claim:nbr} we conclude that the induced subgraph of $G$ on $B$ is a subgraph of the blow-up of a $k$-cycle whose vertices were replaced by the independent sets $B_0,B_1,\dots ,B_{k-1}$. Also, if $v\in V(G)\setminus B$, then there exists $0\leq j\leq k-1$ such that for all $0\leq i\leq k-1$, if $i\not\equiv j\pm 1\pmod{k}$, then $\rho_G(v,B_i)=0$. By the definition of $B$, $\rho_G(v,W_{i})\leq 2/3$ for $i=j+1$ or $i=j-1$. As $|W_i|\geq (1-6k\epsilon_3)|B_i|$ and $2/3+6k\epsilon_3\leq 3/4$, we get
 \begin{align*}
     e_G(v,B_i)&\leq e_G(v,W_i)+\epsilon_3n\leq \frac{2}{3}|W_i|+6k\epsilon_3|B_i|\leq \frac{2}{3}|B_i|+6k\epsilon_3|B_i|\leq \frac{3}{4}|B_i|.
 \end{align*}
 Therefore the sets $B_0,B_1,\dots,B_{k-1}$ satisfy the conditions stated in the lemma.
 \end{proof}

We will use Lemma \ref{lemma:G} to prove that for any $C_{k-2}$-free graph $G$ there is a blow-up of $C_k$ which contains at least as many $k$-cycles as $G$, with equality only if $G$ itself is also a blow-up of $C_k$. This easily implies Theorem \ref{thm:k,k-2} which we restate for convenience.

\mainthm*

\begin{proof} Let $\epsilon_5>0$ be such that $k\epsilon_5<(1/4)/(6k)^{k-1}$. Use Lemma \ref{lemma:G} to obtain $n_0$ and $\delta$, and increase $n_0$ if necessary such that the balanced blow-up of a $k$-cycle on $n\geq n_0$ vertices contains more than $(1-\delta)(n/k)^k$ $k$-cycles.

Let $n\geq n_0$ and let $G$ be a $C_{k-2}$-free graph on $n$ vertices maximising the number of $k$-cycles, so in particular $G$ has more than $(1-\delta)(n/k)^k$ $k$-cycles. Use Lemma \ref{lemma:G} for $\epsilon_5$ and $G$ to obtain $B,B_0,\dots,B_{k-1}$. By symmetry, we can suppose that $B_0$ has minimal size among $B_0,\dots,B_{k-1}$. Let $b=|B|$ and $P=\prod_{i=1}^{k-1}|B_i|$.

The number of $k$-cycles contained in $B$ is at most $P\cdot |B_0|$. The number of $k$-cycles which contain at most $k-2$ vertices from $B$ is at most $k(n-b)^2n^{k-2}$ since any such cycle contains at least two vertices from $V(G)\setminus B$, there are at most $(n-b)^2$ ways to choose two such vertices and there are at most $k$ ways to choose their relative position on the cycle. Let $v\in V(G)\setminus B$. If $v$ has neighbours from at most one of the sets $B_i$, then there is no $k$-cycle containing $v$ and $k-1$ vertices from $B$. If $v$ has neighbours in exactly two of them, they must be in $B_{i+1}$ and $B_{i-1}$ for some $i$ and by symmetry we can assume that $\rho(v,B_{i+1})\leq 3/4$. Then the number of $k$-cycles that contain $v$ and $k-1$ vertices from $B$ is at most $\frac{3}{4}\prod_{j\ne i}|B_{j}|\leq \frac{3}{4}P$. There are $n-b$ possibilities for choosing $v$, so the total number of $k$-cycles in $G$ is at most $P\cdot |B_0|+\frac{3}{4}P(n-b)+k(n-b)^2n^{k-2}$.

Let $H$ be the graph on $n$ vertices obtained by blowing up a $k$-cycle with sets of sizes $|B_0|+n-b,|B_1|,|B_2|,\dots,|B_{k-1}|$. Then $H$ has $(|B_0|+n-b)P$ copies of $C_k$, while $G$ has at most
\begin{align*}
    &P\cdot |B_0|+\frac{3}{4}P(n-b)+k(n-b)^2n^{k-2}\\
    &=(|B_0|+n-b)P+(n-b)\left(k(n-b)n^{k-2}-\frac{1}{4} P\right)\\
    &\leq (|B_0|+n-b)P+(n-b)\left(k\epsilon_5 n^{k-1}-\frac{1}{4}(n/6k)^{k-1}\right).
\end{align*}
By assumption $k\epsilon_5<(1/4)/(6k)^{k-1}$, so if $n-b>0$, then $H$ has more $k$-cycles than $G$. So $n=b$ and $V(G)=B$, therefore $G$ is a subgraph of the blow-up of $C_k$ with sets $B_0,\dots,B_{k-1}$. So $G$ has at most $|B_0||B_1|\dots |B_{k-1}|$ cycles of length $k$ with equality if and only if $G$ is isomorphic to the blow-up of $C_k$ with sets $B_0,\dots,B_{k-1}$. Clearly, $|B_0||B_1|\dots |B_{k-1}|$ is maximal only if $B_0,B_1,\dots,B_{k-1}$ are as equal as possible, completing the proof.\end{proof}

\subsection{Proof of Theorem \ref{thm:k,l}}

In this section, we will prove that if $\ell\geq 3$ is a fixed odd number, then there exists $k_0$, such that for all odd $k\geq k_0$, there is an unbalanced blow-up of an $(\ell+2)$-cycle on $n$ vertices which contains asymptotically more $k$-cycles than the balanced blow-up.

Fix positive real numbers $w_0,w_1,\dots,w_{\ell+1}$ with $w_0+w_1+\dots+w_{\ell+1}=1$ and let $G_n$ be an unbalanced blow-up of an $(\ell+2)$-cycle on $n$ vertices with independent sets $A_0,A_1,\dots,A_{k-1}$ of sizes differing from $w_0n,w_1n,\dots,w_{\ell+1}n$ by at most 1. We want to find a lower bound on the number of $k$-cycles in $G_n$, by only counting the cycles which have one vertex from $A_i$ for each $0\leq i\leq \ell-1$, $(k-\ell)/2$ vertices from $A_{\ell}$ and $(k-\ell)/2$ vertices from $A_{\ell+1}$. There are at least $w_0w_1\dots w_{\ell-1} (w_\ell w_{\ell+1})^{(k-\ell)/2}n^k-o(n^k)$ cycles with this property.

Let $H_n$ be the balanced blow-up of an $(\ell+2)$-cycle on $n$ vertices. The number of $k$-cycles in $H_n$ is bounded from above by the number of walks of length $k-1$ starting from a given independent set in $H_n$. Hence it is bounded by $(n/(\ell+2)+1)\cdot (2n/(\ell+2)+2)^{k-1}=\frac{1}{2}(2/(\ell+2))^kn^k+o(n^k)$.

Therefore it suffices to find weights which satisfy 
\begin{align}\label{eqn:weights}
    w_0w_1\dots w_{\ell-1} (w_\ell w_{\ell+1})^{(k-\ell)/2}>(2/(\ell+2))^k.
\end{align}

If $\ell$ is fixed, then for any weights that satisfy $w_{\ell}w_{\ell+1}>(2/(\ell+2))^2$, there exists $k_0$ such that for all $k\geq k_0$ equation (\ref{eqn:weights}) is satisfied. So choosing for example $w_0=w_1=\dots=w_{\ell-1}=(\ell-2)/(2\ell(\ell+2))$ and $w_{\ell}=w_{\ell+1}=1/4+1/(\ell+2)$, we have for all $k\geq k_0$ that $G_n$ is $C_{\ell}$-free and contains asymptotically more copies of $C_k$ than the balanced blow-up of $C_{\ell+2}$.

\section{Concluding remarks} \label{sec:remarks}

In this section, we show that for any $k\geq 7$, Theorem \ref{thm:k,k-2} is no longer true if $n$ is not assumed to be large enough, in fact it fails for all $k\leq n\leq 3(k-1)/2$. We also show that in Theorem~\ref{thm:k,l} we may take $k=\ell +C\ell/\log \ell$ for some absolute constant $C$.

\begin{proposition}
    If $k\geq 7$ is odd and $k\leq n\leq 3(k-1)/2$, then there is a $C_{k-2}$-free graph on $n$ vertices which contains more $k$-cycles than the balanced blow-up of a $k$-cycle on $n$ vertices.
\end{proposition}

\begin{proof}
    The balanced blow-up of a $k$-cycle on $n$ vertices for $k\leq n\leq 3(k-1)/2$ contains $2^{n-k}$ $k$-cycles. Let $G_n$ be the graph which consist of a $(k-1)$-cycle $v_1v_2\dots v_{k-1}$ and vertices $w_1,w_2\dots, w_{n-k+1}$ such that $w_i$ is connected to $v_{2i-1},v_{2i},v_{2i+1}$, with the indices taken modulo $k-1$. Then $G_n$ is $C_{k-2}$-free, and contains $(n-k+1)2^{n-k+1}>2^{n-k}$ $k$-cycles.
\end{proof}

Note that for $k\leq n\leq k+10$ this example also contains more than $(n/k)^k$ $k$-cycles. (Indeed, $(n/k)^k=(1+\frac{n-k}{k})^k\leq e^{n-k}$, which is less than $(n-k+1)2^{n-k+1}$ for $k\leq n\leq k+10$.)

Given an odd integer $\ell\geq 3$, we want to find a lower bound for $k_0=k_0(\ell)>0$ such that if $k\geq k_0$ is odd, then the balanced blow-up of an $(\ell+2)$-cycle is not maximising the number of $k$-cycles in a $C_\ell$-free graph asymptotically.

\begin{proposition}
    There exists a constant $C$ such that for each odd $\ell \geq 3$ and odd $k\geq \ell+C\ell/\log \ell$, there is an unbalanced blow-up of $C_{\ell+2}$ on $n$ vertices which contains asymptotically more copies of $C_k$ than the balanced blow-up of $C_{\ell+2}$ on $n$ vertices.
\end{proposition}

\begin{proof}
Given $\ell$ and $k$, we first optimize the left-hand side of equation (\ref{eqn:weights}), to find $w_0=w_1=\dots=w_{\ell-1}=1/k$ and $w_{\ell}=w_{\ell+1}=(k-\ell)/2k$. By taking the logarithm in equation (\ref{eqn:weights}), we need

\begin{align}\label{eqn:optw}
    f(k,\ell)=(k-\ell)\log((k-\ell)/2)-k\log(2k/(\ell+2))>0.
\end{align}

By differentiating with respect to $k$ for fixed $\ell$, we get that $f(k,\ell)$ is increasing in $k$ for $k\geq \ell+4+8/(\ell-2)$. So if equation (\ref{eqn:optw}) holds for some $(k_0,\ell)$ with $k_0\geq \ell+4+8/(\ell-2)$, then it holds for any $(k,\ell)$ with $k>k_0$. We will prove that for some $\ell_0$, if $\ell\geq \ell_0$ and $k=\ell+2\ell/\log \ell$, then $f(k,\ell)>0$. Write $t=k-\ell=2\ell/\log \ell$. Using that for all $x>0$ we have $\log(1+x)<x$, we obtain

\begin{align*}
    f(k,\ell)&>t\log\left(\frac{t}{2}\right)-(t+\ell)\left(\log 2+\log\left(1+\frac{t}{\ell}\right)\right)\\
    &>t\log\left(\frac{t}{2}\right)-(t+\ell)(\log 2+t/\ell)\\
    &=t\left(\log t-\log 2-\log 2-t/\ell-(\log 2)\ell/t-1\right).\\
\end{align*}

Choose $\ell_0$ such that $2\log 2+1+2/\log \ell_0\leq 3$. Then, using that $t/\ell=2/\log \ell\leq 2/\log \ell_0$,
\begin{align*}
    f(k,\ell)/t>\log t-3-(\log 2)\ell/t=\log 2 +\log \ell-\log \log \ell -3-(\log \ell)(\log 2)/2
\end{align*}
 For large enough $\ell$ this is positive, so there exists $\ell_0$ such that if $\ell\geq \ell_0$ is odd, and $k\geq \ell+2\ell/\log \ell$ is odd, then the balanced blow-up of an $(\ell+2)$-cycle is not maximising the number of $k$-cycles in a $C_\ell$-free graph asymptotically. Therefore there exists some absolute constant $C$ such that the same holds for any odd $\ell\geq 3$ and $k\geq\ell +C\ell/\log \ell$.
\end{proof}

The following conjecture seems very natural.

\begin{conjecture}
    For any odd integers $k>\ell\geq 3$, there exists $n_0$ such that for all $n\geq n_0$, if $G$ is a $C_{\ell}$-free graph on $n$ vertices maximising the number of $k$-cycles, then $G$ must be a blow-up of an $(\ell+2)$-cycle.
\end{conjecture}

Looking from the other side, it is an interesting problem to determine which blow-up of $C_{\ell+2}$ contains the most $k$-cycles on $n$ vertices for large enough $n$. This is essentially the problem of finding the weights $w_0,w_1\dots,w_{\ell+1}$ as in the proof of Theorem \ref{thm:k,l}, which maximise some cyclic polynomial $P_{k,\ell}$ of degree $k$.

For $k=\ell+2$, $P_{k,\ell}=w_0w_1\dots w_{\ell+1}$ which is maximised if $w_0=w_1=\dots =w_{\ell+1}$ by AM-GM. For $k=\ell+4$, $P_{k,\ell}=w_0w_1\dots w_{\ell+1}(w_0w_1+w_1w_2+\dots +w_\ell w_{\ell+1}+w_{\ell+1}w_0)$, which is again maximised if all weights are equal, as can be seen by using the method of Lagrange multipliers.

\bibliographystyle{abbrv}
\bibliography{bibliography}
 
\end{document}